\documentclass[11pt, oneside]{amsart}
\usepackage[text={5.58in,8.5in},centering,letterpaper,dvips]{geometry}
\usepackage[dvipsnames]{xcolor}
\usepackage{graphicx}
\usepackage{subcaption}
\usepackage{amsfonts}
\usepackage{epsf}
\usepackage{amssymb}
\usepackage{amsmath}
\usepackage{amscd}
\usepackage{tikz}
\usepackage{pdfpages}
\usepackage{fancyhdr}
\usepackage{setspace}
\usepackage{hyperref}
\usepackage[all]{xy}
\usetikzlibrary{matrix}
\usepackage{verbatim}
\usepackage{enumerate}

\theoremstyle{theorem}
\newtheorem{theorem}{Theorem}[section]
\newtheorem{proposition}[theorem]{Proposition}
\newtheorem{lemma}[theorem]{Lemma}
\newtheorem{question}[theorem]{Question}
\newtheorem{corollary}[theorem]{Corollary}

\theoremstyle{definition}

\newtheorem{remark}[theorem]{Remark}

\newcommand{\Q}{\mathbb{Q}}

\newcommand{\wh}[1]{\widehat{#1}}

\newcommand{\A}{\alpha}
\newcommand{\n}{\beta}
\newcommand{\g}{\gamma}
\newcommand{\pd}{\partial}

\newcommand{\X}{\times}

\newcommand{\be}{\begin{enumerate}}
\newcommand{\ee}{\end{enumerate}}

\newcommand{\T}{\mathcal T}
\newcommand{\la}{\langle}

\newcommand{\ov}{\overline}

\makeatletter
\def\@seccntformat#1{%
  \protect\textup{\protect\@secnumfont
    \ifnum\pdfstrcmp{subsection}{#1}=0 \bfseries\fi
    \csname the#1\endcsname
    \protect\@secnumpunct
  }%
}  
\makeatother

\makeatletter
\newtheorem*{rep@theorem}{\rep@title}
\newcommand{\newreptheorem}[2]{%
\newenvironment{rep#1}[1]{%
 \def\rep@title{#2 \ref{##1}}%
 \begin{rep@theorem}}%
 {\end{rep@theorem}}}
\makeatother

\newreptheorem{theorem}{Theorem}
\newreptheorem{lemma}{Lemma}
\newreptheorem{question}{Question}
\newreptheorem{corollary}{Corollary}
\newreptheorem{proposition}{Proposition}

\begin{document}

\rhead{\thepage}
\lhead{\author}
\thispagestyle{empty}


\raggedbottom
\pagenumbering{arabic}
\setcounter{section}{0}


\title{A family of Andrews-Curtis trivializations via 4-manifold trisections}

\author{Ethan Romary}
\address{Berkshire Hathaway Homestate Companies}
\email{eromary@bhhomestate.com}
\urladdr{}

\author{Alexander Zupan}
\address{University of Nebraska-Lincoln, Lincoln, NE 68588}
\email{zupan@unl.edu}
\urladdr{http://www.math.unl.edu/azupan2}

\begin{abstract}
An R-link is an $n$-component link $L$ in $S^3$ such that Dehn surgery on $L$ yields $\#^n(S^1 \X S^2)$.  Every R-link $L$ gives rise to a geometrically simply-connected homotopy 4-sphere $X_L$, which in turn can be used to produce a balanced presentation of the trivial group.  Adapting work of Gompf, Scharlemann, and Thompson, Meier and Zupan produced a family of R-links $L(p,q;c/d)$, where the pairs $(p,q)$ and $(c,d)$ are relatively prime and $c$ is even.  For this family, $L(3,2;2n/(2n+1))$ induces the infamous trivial group presentation $\langle x,y \, | \, xyx=yxy, x^{n+1}=y^n \rangle$, a popular collection of potential counterexamples to the Andrews-Curtis conjecture for $n \geq 3$.  In this paper, we use 4-manifold trisections to show that the group presentations corresponding to a different family, $L(3,2;4/d)$, are Andrews-Curtis trivial for all $d$.
\end{abstract}

\maketitle

\section{Introduction}

The famous \emph{Andrews-Curtis conjecture}~\cite{AC} asserts that any balanced presentation
\[\langle x_1,\dots,x_n \, | \, r_1,\dots,r_n \rangle\]
of the trivial group can be converted to the trivial presentation $\langle x_1,\dots,x_n \, | \, x_1,\dots,x_n \rangle$ by a finite sequence of the following moves:

\begin{enumerate}
\item Replace a relator $r_i$ by $r_i^{-1}$;
\item Replace a relator $r_i$ by $r_ir_j$, where $i \neq j$;
\item Replace a relator $r_i$ by $x_jr_ix_j^{-1}$; and
\item Add or delete a trivial generator/relator pair $x_{n+1}$ and $r_{n+1} = x_{n+1}$.
\end{enumerate}

A presentation $P$ that admits such a trivialization is called \emph{AC-trivial}.  Although the conjecture remains open, there are interesting families of potential counterexamples, many arising from constructions in low-dimensional topology.  Perhaps the best known family in this category is the set of presentations
\[ P_n = \la x,y \, | \, xyx=yxy, x^{n+1}=y^n\rangle,\]
coming from a collection $H_n$ of handle decompositions of the 4-sphere, each with two 1-handles and two 2-handles~\cite{AK,gompf}.  The presentations $P_n$ are not known to be AC-trivial for $n \geq 3$, and they are widely believed to be probable counterexamples to the Andrews-Curtis conjecture.

A related notion is that of an \emph{R-link}, an $n$-component link $L \subset S^3$ such that some Dehn surgery on $L$ yields $\#^n(S^1 \X S^2)$.  Every R-link $L$ naturally gives rise to a balanced presentation $P(L)$ of the trivial group, and in~\cite{GST}, the authors constructed a family of R-links $L_n$ with the property that $P(L_n) = P_n$, the presentations given above.  This construction was generalized by Jeffrey Meier and the second author to produce an R-link $L(p,q;c/d)$ for any co-prime $p$ and $q$ and $c/d \in \Q$ with $c$ even.  With these parameters, $L(3,2;2n/(2n+1))$ is stably equivalent (defined below) to the Gompf-Scharlemann-Thompson links $L_n$~\cite{FHRkS}.  Let $P(p,q;c/d)$ denote the presentation $P(L(p,q;c/d))$.  We have the following natural question:

\begin{question}
Which presentations $P(p,q;c/d)$ can be AC-trivialized?
\end{question}

In~\cite{GST}, the authors showed that the link $L(3,2;0/1)$ is handle-slide equivalent to the unlink, and in forthcoming work, the second author and collaborators show that for all $d$, the links $L(3,2;2/d)$ have the same property, from which it follows that the presentations $P(3,2;0/d)$ and $P(3,2;2/d)$ are AC-trivial~\cite{HNPZ}.  The case $c=4$ is more complicated, and the corresponding question for the links $L(3,2;4/d)$ remains open.  However, in this paper, we prove
 
\begin{theorem}\label{main}
Every presentation of the form $P(3,2;4/d)$ is AC-trivial.
\end{theorem}

The proof breaks into two cases, separated into Proposition~\ref{prop:4n+1} (dealing with the case $d = 4n+1)$ and Proposition~\ref{prop:4n+3} (dealing with the case $d = 4n+3$).  For both proofs, we use trisections of the closed 4-manifolds $X_{L(3,2;4/d)}$ arising from the R-links $L(3,2;4/d)$ in order to construct the presentations $P(3,2;4/d)$.

\begin{remark}
Various sources in the literature differ on whether to allow move (4); in some cases, AC-triviality is defined only with moves (1)-(3), and those sources often use \emph{stable AC-triviality} to allow move (4) as well.  In this paper, AC-triviality will always allow moves (1)-(4).
\end{remark}

\subsection{Organization}

In Section~\ref{sec:prelim}, we establish the necessary background material for the paper.  Section~\ref{sec:4n+1} deals with the first case of the main theorem, while Section~\ref{sec:4n+3} deals with the second case.  We conclude in Section~\ref{sec:quest} with several questions for further investigation.

\subsection{Acknowledgements}
We are grateful to Jeffrey Meier for helpful conversations and for comments on an earlier draft of this paper.  AZ was supported by NSF awards DMS-1664578 and DMS-2005518 during the completion of this work.  

\section{Preliminaries}\label{sec:prelim}

We work in the smooth category throughout.

\subsection{AC-equivalence and automorphisms}

If $P$ and $P'$ are two group presentations related by moves (1)-(4) above, we say that $P$ and $P'$ are \emph{AC-equivalent}, and we write $P \sim P'$.  There is an additional move, the transformation move, that we can apply to a group presentation $P = \langle x_1,\dots, x_n \, | \, r_1,\dots, r_n \rangle$:

\begin{enumerate}
\item[(5)] For an automorphism $\psi$ of the free group $F_n$ generated by $x_1,\dots,x_n$, replace every relator $r_i$ with its image $\psi(r_i)$.
\end{enumerate}

Equivalence of presentations allowing moves (1)-(5) is called \emph{$Q^{**}$-equivalence} \cite{fern,metzler1,metzler2}.  To our knowledge, it remains open whether $Q^{**}$-equivalence is stronger than AC-equivalence; a detailed discussion can be found in Section 3 of~\cite{PU}.  However, the following is known:

\begin{lemma}\label{auto}\cite{BM, mias}
If a presentation $P$ can be converted to the trivial presentation via moves (1)-(3) and (5), then $P$ can be converted to the trivial presentation via moves (1)-(3).
\end{lemma}

As a corollary, we have

\begin{corollary}\label{cor:auto}
If a presentation $P$ is $Q^{**}$-equivalent to the trivial presentation, then $P$ is AC-trivial.
\end{corollary}

\begin{proof}
Suppose that $P = \langle x_1,\dots,x_n \, | \, r_1,\dots, r_n \rangle$ admits a sequence of moves (1)-(5) converting $P$ to the trivial presentation.  Observe that any instances of additions via move (4) can be carried out before any of the other moves, since any automorphism $\psi$ used in move (5) extends by the identity over generators added after the automorphism would have been applied, and instances of deletions via move (4) need not be carried out, because any presentation equivalent to a trivial presentation after a deletion is also equivalent to a trivial presentation of longer length before the deletion.  Thus, it follows that for some $m \geq n$, the presentation
\[ P' = \langle x_1,\dots,x_n , x_{n+1} \dots, x_m\, | \, r_1,\dots, r_n, x_{n+1}, \dots, x_m \rangle\]
can be converted to the trivial presentation by moves of types (1)-(3) and (5).  By Lemma~\ref{auto}, $P'$ is AC-trivial, and since $P$ and $P'$ are related by moves of type (4), it follows that $P$ is AC-trivial as well.
\end{proof}

\subsection{R-links}

As mentioned above, an \emph{R-link} is an $n$-component link $L \subset S^3$ such that $L$ has a Dehn surgery yielding $\#^n(S^1 \X S^2)$.  Every R-link gives rise to a homotopy 4-sphere $X_L$ built with one 0-handle, $n$ 2-handles, $n$ 3-handles, and one 4-handle, where $L$ is the attaching link for the 2-handles, with framings giving by the Dehn surgery coefficients.  Inverting the handle decomposition of $X_L$ yields one 0-handle, $n$ 1-handles, $n$ 2-handles, and one 4-handle, which can be used to produce a balanced presentation for $\pi_1(X_L)$, the trivial group.

In general, an R-link $L$ does not induce a unique such presentation; for instance, a choice of co-cores of the 1-handles in the inverted handle decomposition determines a choice of the generators $x_1,\dots,x_n$ in the presentation $P(L)$, and a different choice induces an automorphism of the free group $F_n$, the fundamental group of the union of the 0-handle and the $n$ 1-handles.  Nevertheless, we can prove the following:

\begin{lemma}\label{lem:q}
Let $L$ be an R-link, and suppose that $P$ and $P'$ are two different presentations induced by $L$.  Then $P$ and $P'$ are $Q^{**}$-equivalent.
\end{lemma}

\begin{proof}
Suppose that $P = \langle x_1,\dots,x_n \, | \, r_1,\dots, r_n\rangle $ and $P'= \langle y_1,\dots y_n \, | \, s_1,\dots, s_n\rangle $ are two presentations induced by $L$.  In the context of the AC-moves, we assume that all groups have the same generating set.  Thus, let $\iota: \langle y_1,\dots,y_n \rangle \rightarrow \langle x_1,\dots,x_n \rangle$ be the map $\iota(y_i) = x_i$, and let $P'' = \iota(P') = \langle x_1,\dots,x_n \,  | \, \iota(s_1),\dots,\iota(s_n) \rangle$, so that $P''$ is identical to $P'$ but uses $x_i$'s instead of $y_i$'s.

There are several sets of choices we make to extract $P$ and $P'$ from $L$:  A choice of co-cores of the 1-handles in the inverted handle decomposition, and a choice of a base point and orientation for each component of the attaching link $L^*$ for the 2-handles in the inverted handle decomposition, since we need a place to begin and a direction when using each component of $L^*$ to read off a relator.  Different choices of orientations yield relators related by a move of type (1).  Likewise, any two choices of base points can be related by cyclic permutations of the corresponding relators, and as such these relators are related by moves of type (3).  Thus, we may assume that $P$ and $P'$ arise from identical choices of base points and orientations of $L^*$.

Regarding the choices of co-cores of the 1-handles, there is a diffeomorphism of the boundary of the union of the 0-handle and 1-handles sending one choice to any other, inducing an automorphism $\sigma$ of the free group.  In other words, each $x_i$ can be expressed as a word in the $y_i$'s, and let $\sigma: \langle x_1,\dots, x_n \rangle \rightarrow \langle y_1,\dots,y_n \rangle$ be the isomorphism such that $\sigma(x_i)$ is the expression of $x_i$ as this word.  Observe that the relators in both presentations are determined by the fixed attaching link $L^*$ with the same base points and orientations, it follows that (possibly after reindexing), we have $\sigma(r_i) = s_i$.

Finally, define $\psi = \iota \circ \sigma$.  Then we have
\[ \langle x_1,\dots,x_n \, | \, \psi(r_1), \dots \psi(r_n) \rangle =\langle x_1,\dots,x_n \, | \, \iota(s_1), \dots \iota(s_n) \rangle = P''.\]

We conclude that $P$ and $P''$ are related by a move of type (5), and thus any two such presentations $P$ and $P''$ are $Q^{**}$-equivalent.
\end{proof}

In view of Lemma~\ref{lem:q}, the $Q^{**}$-equivalence class of the the presentation induced by the R-link $L$ is well-defined, and so we use $P(L)$ to denote this equivalence class.  By Corollary~\ref{cor:auto}, if we can show that some presentation in the equivalence class $P(L)$ is AC-trivial, then every presentation in $P(L)$ is AC-trivial.  For this reason and for the purposes of proving AC-triviality, we often blur the distinction and abuse notation to let $P(L)$ denote any representative of the equivalence class $P(L)$.

There are also moves on the R-link $L$ that leave $P(L)$ invariant:  If $L$ and $L'$ are related by a sequence of handle-slides, we say $L$ and $L'$ are \emph{handle-slide equivalent}.  More generally, if $L$ and $L'$ are R-links (possibly with different numbers of components) and $U$ and $U'$ are unlinks, such that the split links $L \sqcup U$ and $L' \sqcup U'$ are handle-slide equivalent, we say that $L$ and $L'$ are \emph{stably equivalent}.  We have the following well-known lemma.

\begin{lemma}\label{lem:stable}
If $L$ and $L'$ are stably equivalent, then $P(L) = P(L')$.
\end{lemma}

\begin{proof}
Any handle-slide of $L$ over $L'$ induces a handle-slide of $(L')^*$ over $L^*$, the dual attaching links in the inverted handle decompositions.  As such, the corresponding presentations can be related by moves of type (2) (and possibly other moves corresponding to the choices referenced in the proof of Lemma~\ref{lem:q}).  If $L$ and $L'$ are stably handle-slide equivalent, then there are unlinks $U$ and $U'$ such that $L \sqcup U$ and $L' \sqcup U'$, are handle-slide equivalent, so that $P(L \sqcup U) = P(L' \sqcup U')$.  In this case, $P(L)$ and $P(L \sqcup U)$ are related by moves of type (4); thus, $P(L) = P(L \sqcup U)$.  Similarly, $P(L') = P(L' \sqcup U')$, completing the proof.
\end{proof}

The generalized Property R conjecture (GPRC) asserts that every R-link $L$ is handle-slide equivalent to an unlink, and the stable version of the GPRC asserts that $L$ is stably handle-slide equivalent to an unlink.  Both conjectures, if true, would imply that every presentation $P(L)$ arising in this way is AC-trivial.  For a detailed discussion of R-links and the AC-conjecture, the reader is encouraged to refer to~\cite{GST,MSZ}. 

\subsection{The family $L(3,2;c/d)$}

In~\cite{MZDehn}, Meier and the second author used work in~\cite{GST} and~\cite{Schar} to introduce the family $L(3,2;c/d)$ of R-links, and in~\cite{FHRkS}, they extended the construction to $L(p,q;c/d)$.  The construction is described in much greater detail in~\cite{MZDehn} and~\cite{FHRkS} but we briefly summarize here:  Let $Q$ be the square knot $3_1 \# \ov{3_1}$, and let $F$ be the fiber for $Q$ in $S^3$, a surface with genus two and one boundary component.  Then $S^3_0(Q)$, the closed 3-manifold resulting from 0-surgery on $Q$, is fibered with fiber $\wh F$, the closed genus two surface obtained by capping off the boundary component of $F$ with a disk.  Viewing $\wh F$ as the quotient of an annulus with hexagonal boundary components identified in opposite pairs as shown in Figure~\ref{fig:2} below, the \emph{closed monodromy} $\wh \varphi: \wh F \rightarrow \wh F$ associated to $S^3_0(Q)$ is a clockwise rotation of $\pi/3$ radians.

In addition, if $S$ represents the 2-sphere with four cone points of order three, there is a branched covering map $\rho:\wh F \rightarrow S$ with the property that $\rho \circ \wh \varphi = \rho$.  Curves in $S$ that avoid the cone points are parametrized by the extended rational numbers $\overline{\Q}$, and for $c$ even, the curve $\lambda_{(c/2)/d}$ corresponding to the rational number $(c/2)/d$ lifts via $\rho$ to three curves $V_{c/d}$, $V'_{c/d}$, and $V''_{c/d}$ contained in $\wh F$ and permuted by $\wh \varphi$.  These curves can be chosen to lie in $F$, and as such the link $Q \cup V_{c/d} \cup V'_{c/d} \cup V''_{c/d} \subset S^3$ is an R-link, stably handle-slide equivalent to any of its two component sublinks.  In~\cite{FHRkS}, the link $L(3,2;c/d)$ is defined to be $V_{c/d} \cup V'_{c/d}$, but for the purposes of determining whether $P(3,2;c/d)$ is AC-trivial, we can use any of these links by Lemmas~\ref{lem:q} and~\ref{lem:stable}

\begin{remark}
The convention here for $c/d$ agrees with~\cite{FHRkS} but differs from~\cite{MZDehn} and~\cite{Schar} in that the numerator $c$ is doubled in our setting.  This doubling is explained in detail in Remarks 4.1 and 4.12 of~\cite{FHRkS}.
\end{remark}

\begin{remark}
It was proved in~\cite{Schar} that the links $L(3,2;2n/(2n+1))$ are stably handle-slide equivalent to the links $L_n$ appearing in~\cite{GST}, and thus $P(3,2;2n/(2n+1))$ is equivalent to the famous presentation $P_n$ described in the introduction.
\end{remark}

\subsection{Trisecting $X_{L(3,2;c/d)}$}

Gay and Kirby introduced 4-manifold trisections in~\cite{GK}.  A $(g;k_1,k_2,k_3)$-\emph{trisection} $\mathcal{T}$ of a closed, smooth 4-manifold $X$ is a decomposition $X = X_1 \cup X_2 \cup X_3$ with the properties

\begin{enumerate}
\item Each $X_i$ is a 4-dimensional 1-handlebody with $\text{rk}(\pi_1(X_i)) = k_i$;
\item Each $H_i = X_{i-1} \cap X_{i}$ is a 3-dimensional genus $g$ handlebody; and
\item $\Sigma = X_1 \cap X_2 \cap X_3$ is a genus $g$ surface.
\end{enumerate}

A trisection $\mathcal{T}$ is determined by the union $H_1 \cup H_2 \cup H_3$, which is in turn determined by a collection of three cut systems, $\A$, $\n$, and $\g$, contained in $\Sigma$, called the \emph{central surface} of the trisection.  The triple $(\A,\n,\g)$ is called a \emph{trisection diagram}.

We have already discussed the closed monodromy $\wh\varphi: \wh F \rightarrow \wh F$ for the square knot $Q$, but in this setting, we will need to use the monodromy $\varphi: F \rightarrow F$ for $Q$, which is required to be the identity on $\pd F = Q$.  The monodromy $\varphi$ consists of a $\pi/3$ rotation as before, but this time followed by an isotopy that drags the boundary component of $F$ back to where it started.  The curves $V_{0/1}$, $V'_{0/1}$, and $V''_{0,1}$ and their images under $\varphi$ are shown in Figure~\ref{fig:2}.  (See also Figure 9 of~\cite{Schar}.)

\begin{figure}[h!]
\begin{subfigure}{.48\textwidth}
  \centering
  \includegraphics[width=1\linewidth]{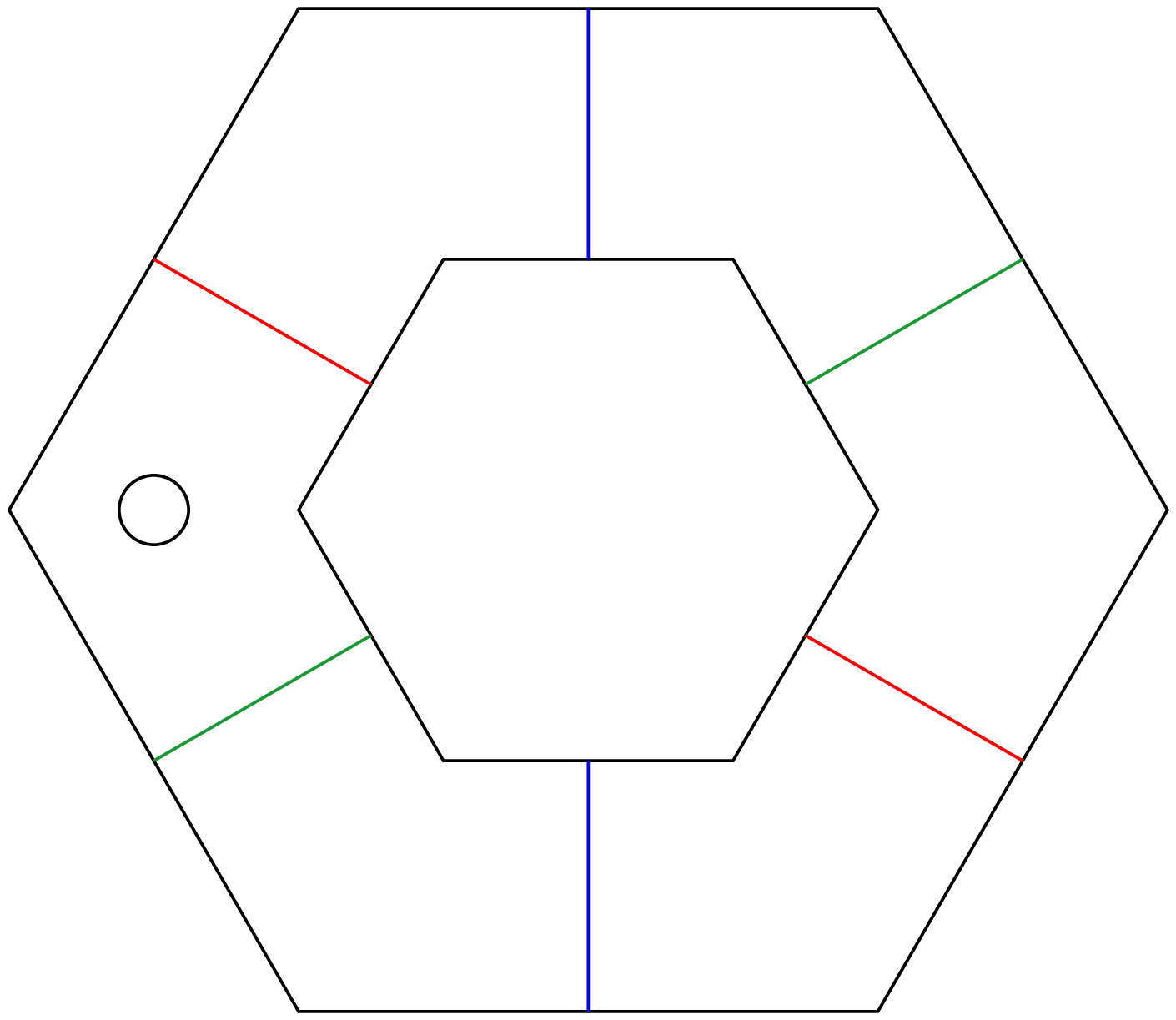}
  \label{fig:braid1}
\end{subfigure}
\begin{subfigure}{.48\textwidth}
  \centering
  \includegraphics[width=1\linewidth]{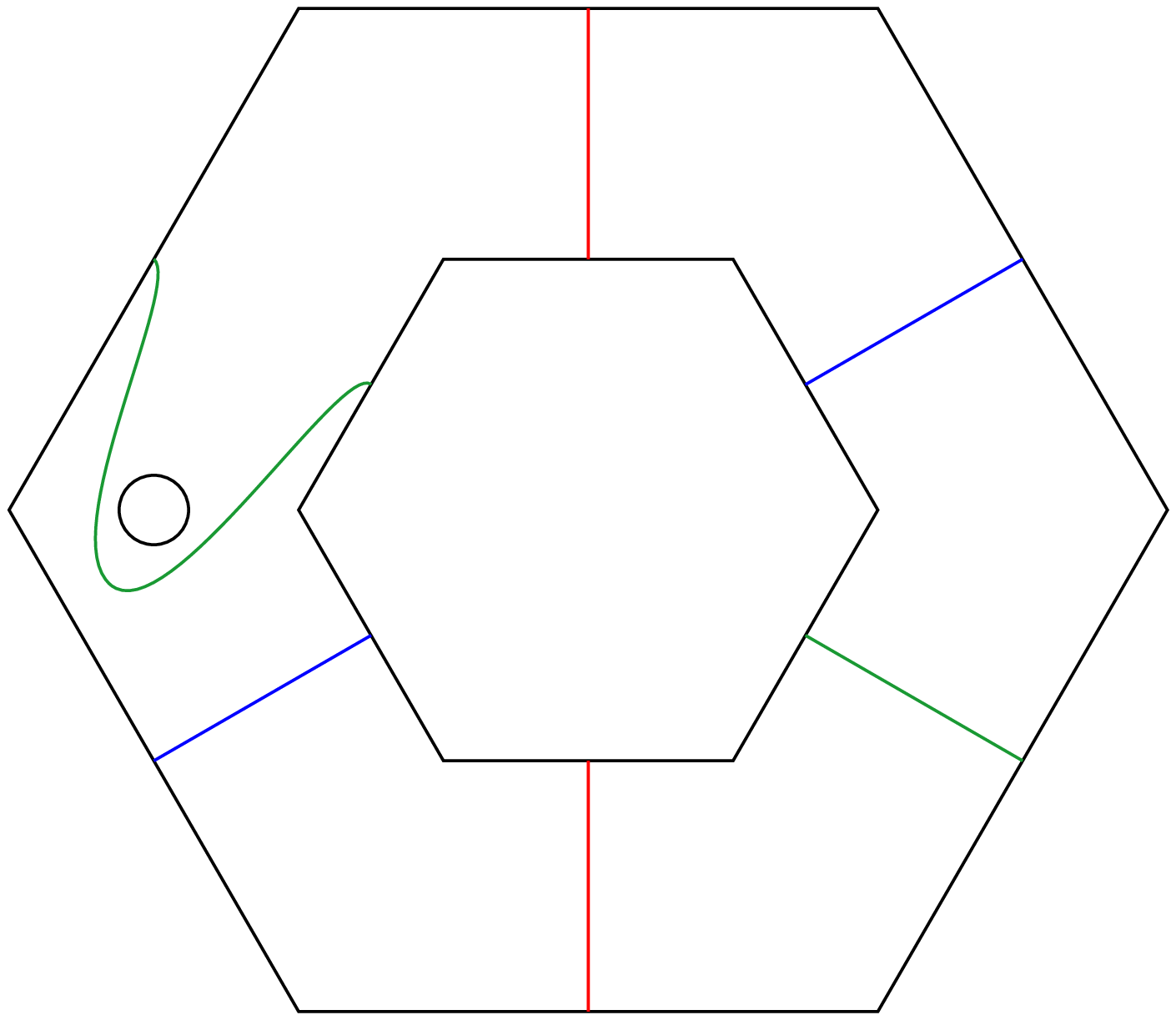}
  \label{fig:braid2}
\end{subfigure}
  \caption{At left, curves $V_{0/1}$ (red), $V'_{0/1}$ (blue), and $V''_{0/1}$ (green).  At right, their images under the monodromy $\varphi$.  Note that while the closed monodromy $\wh \varphi$ permutes the three curves, the monodromy $\varphi$ does not.}
\label{fig:2}
\end{figure}

Next, we define three cut systems, which will determine a trisection diagram for a trisection $\T(c/d)$ for $X_{L(3,2;c/d)}$.  Define $\Sigma = \pd(F \X I)$, where $\pd F \X I$ has been crushed to the single curve $\pd F$, so that we may view $\Sigma$ as $F\# \overline{F}$.  For a curve or arc $a$ embedded in $F$, let $\overline{a}$ denote the mirror image of $a$ contained in $\overline{F}$.  Let $a_1,a_2,a_3,a_4$ be four pairwise disjoint arcs in $F$ cutting $F$ into a disk, and define
\begin{eqnarray*}
\A &=& \{\varphi(a_1) \cup \overline{a_1}, \,\varphi(a_2) \cup \overline{a_2}, \,\varphi(a_3) \cup \overline{a_3},\, \varphi(a_4) \cup \overline{a_4}\} \\
\n &=& \{ a_1 \cup \overline{a_1},\,a_2 \cup \overline{a_2},\,a_3 \cup \overline{a_3},\, a_4 \cup \overline{a_4}\} \\
\g &=& \{V_{c/d},\,V'_{c/d},\,\overline{V_{c/d}},\,\overline{V'_{c/d}}\}.
\end{eqnarray*}

Then we have the following, which is Proposition 19 from~\cite{MZDehn}.

\begin{proposition}\label{diagram}
The triple $(\A,\n,\g)$ is a $(4;0,2,2)$-trisection diagram for a trisection $\T(c/d)$ of $X_{L(3,2;c/d)}$.
\end{proposition}

Finally, we will need a tool which we can use to extract a handle decomposition from a trisection, a restatement of parts of Lemma 13 from~\cite{GK}.

\begin{lemma}\label{handle}
Suppose $X = X_1 \cup X_2 \cup X_3$ is a $(g;k_1,k_2,k_3)$-trisection, with $H_i = X_{i-1} \cap X_i$.  Then $X$ has a handle decomposition with
\begin{enumerate}
\item one 0-handle and $k_1$ 1-handles (contained in $X_1$),
\item $g-k_2$ 2-handles (contained in $X_2)$, and
\item $k_3$ 3-handles (contained in $X_3$).
\end{enumerate}
In addition, a choice of $k_1$ pairwise disjoint and mutually nonseparating curves $C$ in $\Sigma$ bounding disks in both $H_1$ and $H_2$ represent the intersections of $k_1$ co-cores of the 1-handles with $\Sigma$.  Finally, an attaching link $L$ for the 2-handles can be obtained by choosing $g-k_2$ curves bounding disks in $H_3$ that are dual to $g-k_2$ curves bounding disks in $H_2$, and viewing $L$ as a framed link (with framing given by the surface $\Sigma)$ in the 3-manifold $\pd X_1$.
\end{lemma}

The lemma can also be applied by any permutation of $\{1,2,3\}$ to the indices of the components $X_i$.  An example that carries out this procedure appears in Subsection 2.6 of~\cite{MZB}.

\section{The case $d = 4n+1$}\label{sec:4n+1}

We break the proof of the main theorem into two cases.  First, in this section we consider $c/d$ of the form $4/(4n+1)$.  In the next section, we examine $c/d$ of the form $4/(4n+3)$.  The proofs are quite similar but the specific curves and computations are different.  For the remainder of this section, we will label curves and arcs in the surface $\overline{F}$ at right in Figure~\ref{fig:4n+1} as follows:

\begin{enumerate}
\item The red arc is $\overline{a_1}$, and the pink arc is $\overline{a_2}$.
\item The dark blue arc is $\overline{b_1}$, and the light blue arc is $\overline{b_2}$.
\item The dark green curve is $\overline{V_{4/1}}$ and the light green curve is $\overline{V'_{4/1}}$.
\end{enumerate}

\begin{figure}[h!]
\begin{subfigure}{.48\textwidth}
  \centering
  \includegraphics[width=1\linewidth]{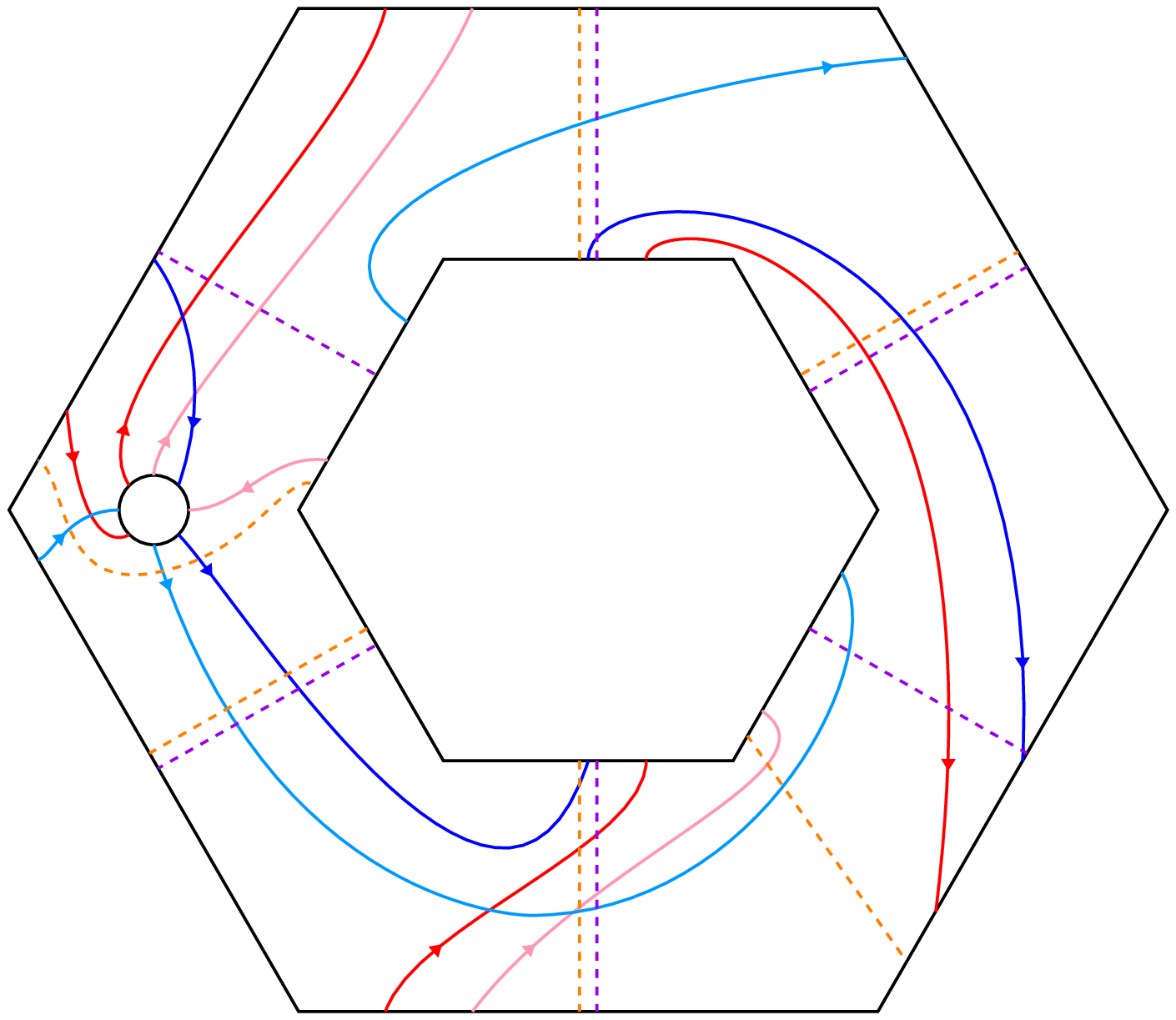}
  \label{fig:braid1}
  \caption{$F$}
\end{subfigure}
\begin{subfigure}{.48\textwidth}
  \centering
  \includegraphics[width=1\linewidth]{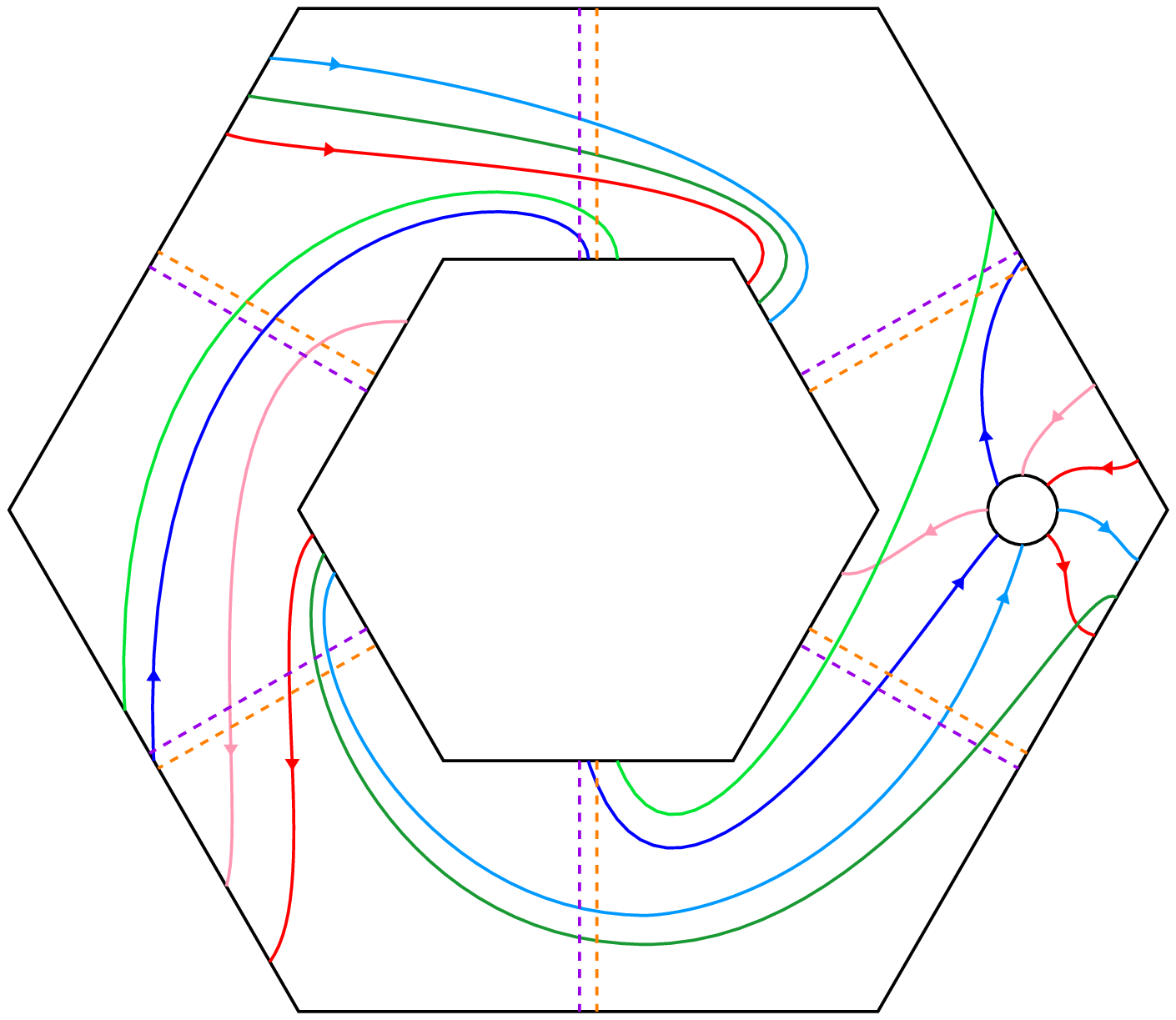}
  \label{fig:braid2}
  \caption{$\ov{F}$}
\end{subfigure}
	\caption{Curves and arcs in $\Sigma = F \# \ov{F}$ used to compute $P(3,2;4/(4n+1))$}
\label{fig:4n+1}
\end{figure}

In addition, we let $a_3$ and $a_4$ denote the red and pink arcs, respectively, at left in Figure~\ref{fig:4n+1}, noting that $a_3 = \varphi(a_1)$, $a_4 = \varphi(a_2)$, and the dark blue and light blue arcs in $F$ are the arcs $b_1$ and $b_2$.  We also let $\tau:F \rightarrow F$ to be the product of right-handed Dehn twists about the pairwise disjoint collection of curves $V_{0/1} \cup V'_{0/1} \cup V''_{0/1}$, so that $\ov{\tau}:\ov{F} \rightarrow \ov{F}$ is the product of left-handed Dehn twists about $\overline{V_{0/1}} \cup \overline{V'_{0/1}} \cup \overline{V''_{0/1}}$, shown as dotted purple curves in Figure~\ref{fig:4n+1}.  Define $V^*_{0/1}$ be the curve obtained by sliding $V_{0/1}$ over $Q$ (equivalently, $V^*_{0/1}$ is the image of $V''_{0/1}$ under $\varphi$ as in Figure~\ref{fig:2}), and let $\tau_{*}: F \rightarrow F$ be identical to $\tau$ except for replacing the Dehn twist about $V_{0/1}$ with a Dehn twist about $V^*_{0/1}$, shown in dotted orange in the same figure.  Since $\varphi(V_{0/1} \cup V'_{0/1} \cup V''_{0/1}) = V'_{0/1} \cup V''_{0/1} \cup V^*_{0/1}$, it follows from Section 3.5 of~\cite{primer} that
\begin{equation}\label{eqn:comm}
\varphi \circ \tau = \tau_* \circ \varphi.
\end{equation}

In Lemmas 4.7 and 4.8 of~\cite{FHRkS}, it was shown that

\begin{lemma}\label{twist}
For the curves $V_{4/d}$ and $V'_{4/d}$ in $\Sigma$, we have
\[ \tau^n(V_{4/d}) = V_{4/(d+4)} \quad \text{ and } \quad  \tau^n(V'_{4/d}) = V'_{4/(d+4)}.\]
\end{lemma}

Using the symmetry of $\Sigma$, it follows that
\[ \ov{\tau}^n(\ov{V_{4/d}}) =\ov{V_{4/(d+4)}}\quad \text{ and } \quad  \ov{\tau}^n(\ov{V'_{4/d}}) = \ov{V'_{4/(d+4)}}\]
as well.  This relationship is precisely why we need to address two cases, when $d= 4n+1$ and when $d = 4n+3$.  The observant reader should note that $\tau$ has the property that $\tau^{-2} = \tau_0$, where $\tau_0$ is the left-handed multi-twist along three pairs of parallel curves defined in~\cite{FHRkS}.

\begin{lemma}\label{lem:4n+1}
The trisection $\T(4/(4n+1))$ gives rise to an inverted handle decomposition of $X_{L(3,2;4/(4n+1))}$ with two 1-handles and two 2-handles, where co-cores of the 1-handles meet the central surface $\Sigma$ in the curves $\tau^n(b_1) \cup  \ov{\tau}^n(\overline{b_1})$ and $\tau^n(b_2) \cup  \ov{\tau}^n(\overline{b_2})$, while an attaching link for the 2-handles is determined by $\tau_*^n(a_3) \cup \ov{\tau}^n(\overline{a_1})$ and $\tau_*^n(a_4) \cup \ov{\tau}^n(\overline{a_2})$.  The corresponding group presentation $P(3,2;4/(4n+1))$ is
\[ P(3,2;4/(4n+1)) = \langle x,y \,|\,  \overline{x}(y\overline{x})^n \overline{y}(\overline{y}x)^n y, \overline{x}(y\overline{x})^n \overline{y}(x\overline{y})^n \rangle.\]
\end{lemma}

\begin{proof}
By Lemma~\ref{diagram}, there exists a trisection diagram $(\A,\n,\g)$ for $\T(4n/(4n+1))$ such that $\A$ contains for $i=1,2$ the curves $\A_i = \varphi(\tau^n(a_i)) \cup \ov{\tau^n(a_i)} = \varphi(\tau^n(a_i)) \cup \ov{\tau}^n(\ov{a_i})$, which bound disks in $H_1$.  Applying Equation~\ref{eqn:comm} repeatedly, we have
\[ \tau_*^n(a_3) \cup \ov{\tau}^n(\overline{a_1}) = \tau_*^n(\varphi(a_1)) \cup \ov{\tau}^n(\overline{a_1}) = \varphi(\tau^n(a_1)) \cup \ov{\tau}^n(\overline{a_1}) = \A_1;\]
\[ \tau_*^n(a_4) \cup \ov{\tau}^n(\overline{a_2}) = \tau_*^n(\varphi(a_2)) \cup \ov{\tau}^n(\overline{a_2}) = \varphi(\tau^n(a_2)) \cup \ov{\tau}^n(\overline{a_2}) = \A_2.\]
Additionally, for $i=1,2$, the curves $\n_i = \tau^n(b_i) \cup \ov{\tau^n(b_i)} =  \tau^n(b_i)  \cup  \ov{\tau}^n(\overline{b_i})$ are curves in $\n$, which bound in $H_2$, and $\n_1$ and $\n_2$ are disjoint from the curves in $\g = \{\g_1,\g_2,\g_3,\g_3\}$, the curves
\[ \{V_{4/(4n+1)},V'_{4/(4n+1)},\overline{V_{4/(4n+1)}},\overline{V_{4/(4n+1)}}\} = \{ \tau^n(V_{4/1}),\tau^n(V'_{4/1}), \ov{\tau}^n(\overline{V_{4/1}}),\ov{\tau}^n(\overline{V'_{4/1}})\},\]
as asserted by Lemma~\ref{twist}.  Thus, the curves $\n_1$ and $\n_2$ bound disks in the handlebody $H_3$ as well.

Now, observe that $\A_1$ meets $\g_1$ once and avoids $\g_2$, while $\A_2$ meets $\g_2$ once and avoids $\g_1$.  By Lemma 14 of~\cite{GK}, the handle decomposition of $X_{L(3,2;4/(4n+1))}$ determined by $L(3,2;4/(4n+1))$ is compatible with the handle decomposition determined by $\T(4/(4n+1))$, with no 1-handles (contained in the 4-ball $X_1)$, two 2-handles (contained in $X_3$), and two 3-handles (contained in $X_2$).  It follows that in the inverted handle decomposition, the two 1-handles are in $X_2$, while the two 2-handles can be viewed in $X_3$.  Since $\pd X_2 = H_2 \cup H_3$, determined by $\n$ and $\g$, the curves $\n_1$ and $\n_2$ above can be chosen as the intersections of co-cores of the 1-handles with $\Sigma$.  Moreover, the curves $\A_1$ and $\A_2$ are dual to $\g_1$ and $\g_2$, so that an attaching link for the 2-handles is $\A_1 \cup \A_2$ by Lemma~\ref{handle}.

Finally, using Figure~\ref{fig:4n+1} with orientation as shown, we can read off the relators determined by $\A_1$ and $\A_2$.  To compute oriented intersections, we find the sign of each point of $\A_i \cap \n_j$, where the direction of $\A_i$ is the first element of a standardly oriented ordered basis.  Note that the multi-twists $\tau$ and $\tau_*$ differ only at the curves $V_{0/1}$ (which twists $\n_i$ but not $\A_i$) and $V^*_{0/1}$ (which twists $\A_i$ but not $\n_i$).  We have included an illustration of the twisting in a neighborhood of these curves in Figure~\ref{fig:twist1} to aid in our computation.

We use the generators $x$ and $y$ for $\n_1$ and $\n_2$, respectively, and we read the relators $r$ and $s$ from $\A_1$ and $\A_2$, respectively, starting in $F$ at $Q$ and following the orientations (note that there are no intersections to be seen in $\overline{F}$).  We have
\begin{eqnarray*}
r &=& \overline{x}(y\overline{x})^n \overline{y}(\overline{y}x)^n y; \\
s &=& \overline{x}(y\overline{x})^n \overline{y}(xy\overline{y}\overline{y})^n = \overline{x}(y\overline{x})^n \overline{y}(x\overline{y})^n
\end{eqnarray*}
\end{proof}

Next, we trivialize our computed presentation.

\begin{figure}[h!]
\begin{subfigure}{.48\textwidth}
  \centering
  \includegraphics[width=.25\linewidth]{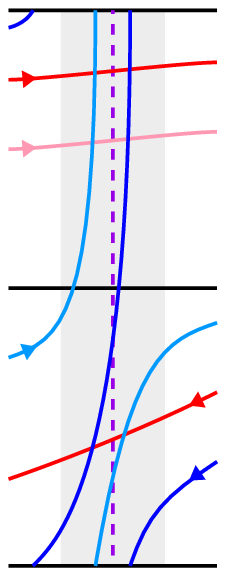}
  \label{fig:braid1}
  \caption{Dehn twisting $b_1$ and $b_2$ about $V_{0/1}$}
\end{subfigure}
\begin{subfigure}{.48\textwidth}
  \centering
  \includegraphics[width=.25\linewidth]{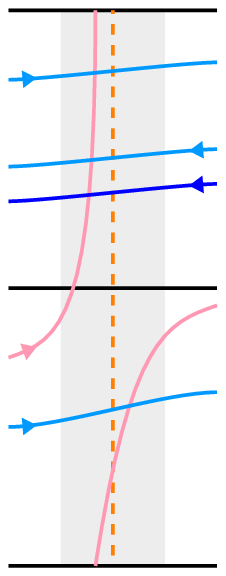}
  \label{fig:braid2}
  \caption{Dehn twisting $a_3$ and $a_4$ about $V^*_{0/1}$}
\end{subfigure}
	\caption{Images of Dehn twists to aid in computations for Lemma~\ref{lem:4n+1}.}
\label{fig:twist1}
\end{figure}

\begin{proposition}\label{prop:4n+1}
The presentation $P(3,2;4/(4n+1))$ is AC-trivial.
\end{proposition}

\begin{proof}
Let $(r_0,s_0) = (r,s)$ be the relators from Lemma~\ref{lem:4n+1}.  First, let $r_1 = r_0^{-1}$ so that
\[ r_1 = \ov{y}(\ov{x}y)^ny(x\ov{y})^nx.\]
Next, let $r_2 = r_1s_0$, yielding
\[ r_2 = (\ov{y}(\ov{x}y)^ny(x\ov{y})^nx)( \overline{x}(y\overline{x})^n \overline{y}(x\overline{y})^n) = \ov{y}(\ov{x}y)^n(x\ov{y})^n.\]
Regroup the terms in $s_0$ to get
\[ s_0 = \overline{x}(y\overline{x})^n \overline{y}(x\overline{y})^n = \ov{x}(y\ov{x})^n(\ov {y} x)^n \ov{y}.\]
Now, cyclically permute $s_0$ to get 
\[ s_1 = (y \ov{x})^n(\ov{y}x)^n \ov{y}\ov{x},\]
and then let $r_3 = r_2s_1$.  Thus,
\[ r_3 = (\ov{y}(\ov{x}y)^n(x\ov{y})^n)(y \ov{x})^n((\ov{y}x)^n \ov{y}\ov{x}) = \ov{y}^2\ov{x}.\]
At this point, $r_3$ is the relator $x = \ov{y}^2$, and we can use cyclic permutations and substitutions to convert $s_1$ to $s_2 = y$, which in turn lets us convert $r_3$ to $r_4 = x$, completing the proof.
\end{proof}

\section{The case $d = 4n+3$}\label{sec:4n+3}

We proceed in a manner nearly identical to that of Section~\ref{sec:4n+1}, but starting with different curves and arcs in $\Sigma$.  For the remainder of this section, we will label curves and arcs in the surface $\overline{F}$ at right in Figure~\ref{fig:4n+3} as follows:

\begin{enumerate}
\item The red arc is $\overline{a_1}$, and the pink arc is $\overline{a_2}$.
\item The dark blue arc is $\overline{b_1}$, and the light blue arc is $\overline{b_2}$.
\item The dark green curve is $\overline{V_{4/3}}$ and the light green curve is $\overline{V'_{4/3}}$.
\end{enumerate}

\begin{figure}[h!]
\begin{subfigure}{.48\textwidth}
  \centering
  \includegraphics[width=1\linewidth]{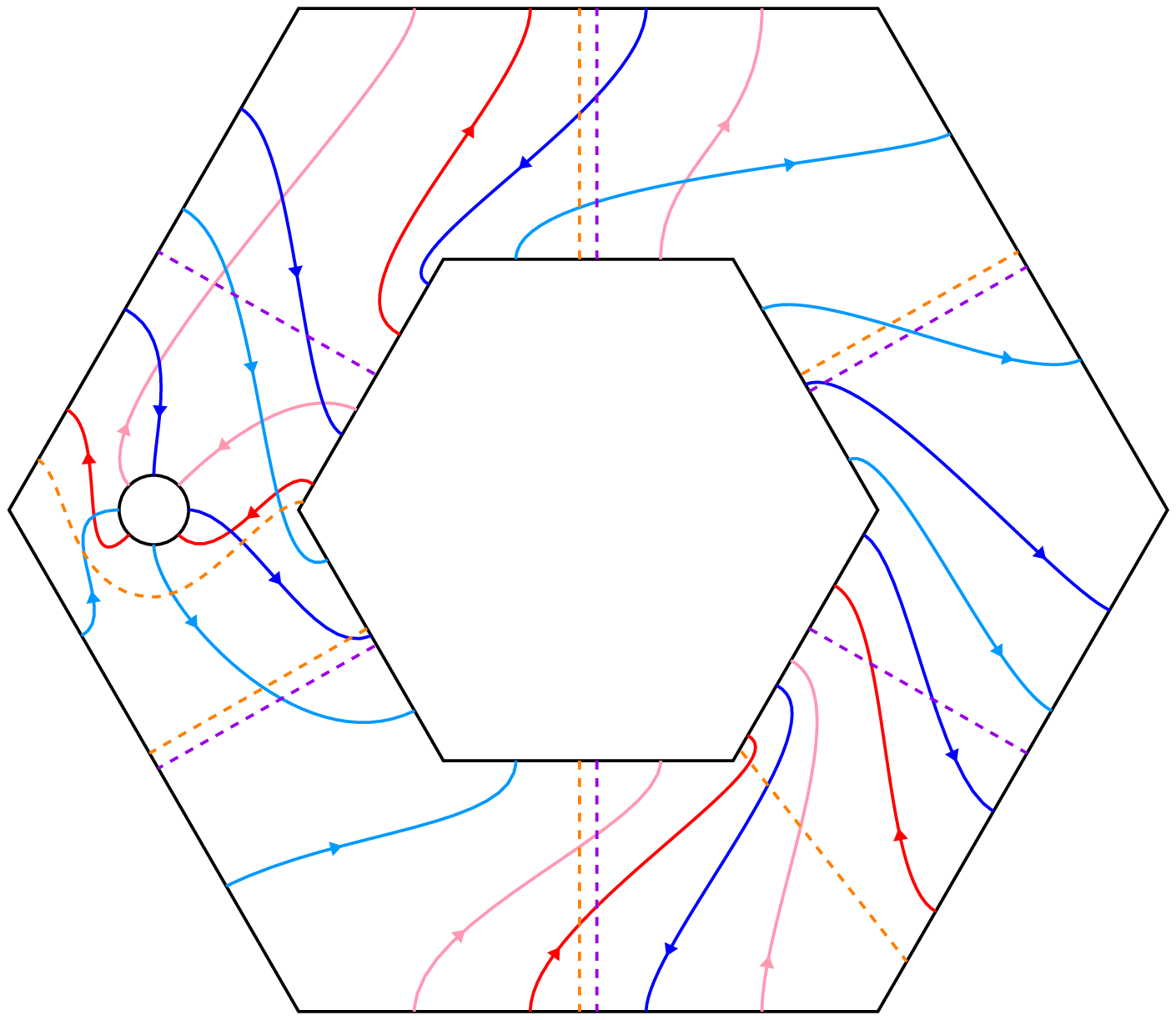}
  \label{fig:braid1}
  \caption{$F$}
\end{subfigure}
\begin{subfigure}{.48\textwidth}
  \centering
  \includegraphics[width=1\linewidth]{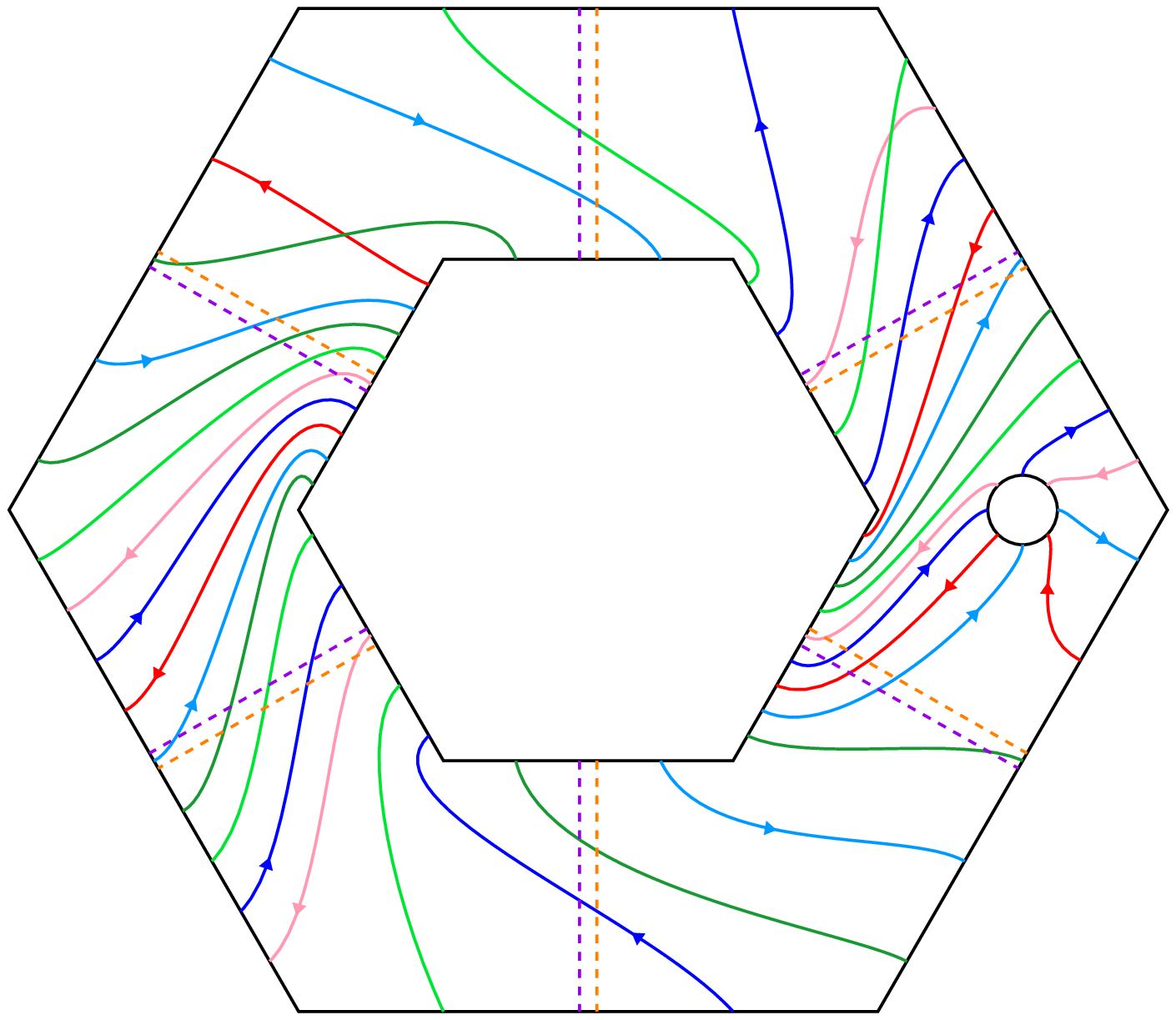}
  \label{fig:braid2}
  \caption{$\ov{F}$}
\end{subfigure}
	\caption{Curves and arcs in $\Sigma = F \# \ov{F}$ used to compute $P(3,2;4/(4n+3))$}
\label{fig:4n+3}
\end{figure}

In addition, we let $a_3$ and $a_4$ denote the red and pink arcs at left in Figure~\ref{fig:4n+3}; as before, $a_3 = \varphi(a_1)$, $a_4 = \varphi(a_2)$, and the dark blue and light blue arcs in $F$ are the arcs $b_1$ and $b_2$.  Recall the definitions of $\tau$ and $\tau_*$ from Section~\ref{sec:4n+1}.

\begin{lemma}\label{lem:4n+3}
The trisection $\T(4/(4n+3))$ gives rise to an inverted handle decomposition of $X_{L(3,2;4/(4n+3))}$ with two 1-handles and two 2-handles, where co-cores of the 1-handles meet the central surface $\Sigma$ in the curves $\tau^n(b_1) \cup  \ov{\tau}^n(\overline{b_1})$ and $\tau^n(b_2) \cup  \ov{\tau}^n(\overline{b_2})$, while an attaching link for the 2-handles is determined by $\tau_*^n(a_3) \cup \ov{\tau}^n(\overline{a_1})$ and $\tau_*^n(a_4) \cup \ov{\tau}^n(\overline{a_2})$.  The corresponding group presentation $P(3,2;4/(4n+3))$ is
\[ P(3,2;4/(4n+3)) = \langle x,y \,|\, \ov{y}(\ov{x} \ov{y} \ov{x})^n (yx^2)^n yx, \, \ov{x}(\ov{y}\ov{x}^2)^n \ov{y} \ov{x}\ov{y}(xyx)^nxy \rangle\]
\end{lemma}

\begin{proof}
As in the proof of Lemma~\ref{lem:4n+1}, there is a trisection diagram $(\A,\n,\g)$ for $\T(4n/(4n+3))$ that includes (for $i=1,2$) $\A_i = \varphi(\tau^n(a_i)) \cup \ov{\tau}^n(\ov{a_i})$ as curves in $\A$ bounding disks in $H_1$, and by the same argument, we have
\[ \tau_*^n(a_3) \cup \ov{\tau}^n(\overline{a_1}) = \A_1;\]
\[ \tau_*^n(a_4) \cup \ov{\tau}^n(\overline{a_2}) = \A_2.\]
For $i=1,2$, the curves $\n_i = \tau^n(b_i) \cup  \ov{\tau}^n(\overline{b_i})$ are curves in $\n$, disjoint from the curves in $\g = \{\g_1,\g_2,\g_3,\g_3\}$, the curves
\[ \{ \tau^n(V_{4/3}),\tau^n(V'_{4/3}), \ov{\tau}^n(\overline{V_{4/3}}),\ov{\tau}^n(\overline{V'_{4/3}})\},\]
as asserted by Lemma~\ref{twist}.  Thus, $\n_1$ and $\n_2$ bound disks in the handlebody $H_3$ as well.

Since $\A_1$ meets $\g_1$ once and avoids $\g_2$, while $\A_2$ meets $\g_2$ once and avoids $\g_1$, the same argument as in Lemma~\ref{lem:4n+1} can be used to show that in the inverted handle composition coming from $L(3,2;4/(4n+3))$, co-cores of the 1-handles meet $\Sigma$ in $\n_1$ and $\n_2$, and an attaching link for the 2-handles consists of $\A_1$ and $\A_2$.  Now, using Figures~\ref{fig:4n+3} and~\ref{fig:twist2}, we can read off the relators $r$ and $s$ (following orientations and starting in $F$ at $Q$),

\begin{eqnarray*}
r &=& \ov{y}(\ov{x} \ov{y} \ov{x})^n (yxy \ov{y} x)^n yx = \ov{y}(\ov{x} \ov{y} \ov{x})^n (yx^2)^n yx\\
s &=& \ov{x}(\ov{y}\ov{x}^2)^n \ov{y} \ov{x}\ov{y}(xyxy\ov{y})^nxy = \ov{x}(\ov{y}\ov{x}^2)^n \ov{y} \ov{x}\ov{y}(xyx)^nxy
\end{eqnarray*}

\end{proof}

\begin{figure}[h!]
\begin{subfigure}{.48\textwidth}
  \centering
  \includegraphics[width=.25\linewidth]{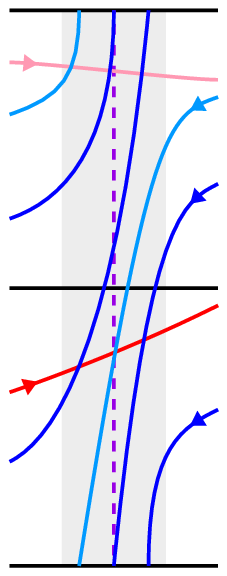}
  \label{fig:braid1}
  \caption{Dehn twisting $b_1$ and $b_2$ about $V_{0/1}$}
\end{subfigure}
\begin{subfigure}{.48\textwidth}
  \centering
  \includegraphics[width=.25\linewidth]{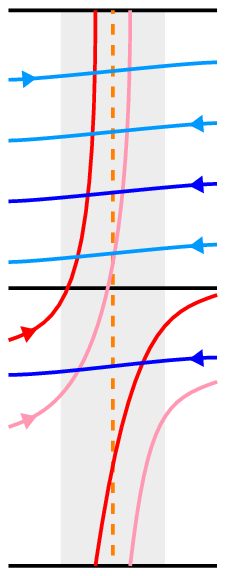}
  \label{fig:braid2}
  \caption{Dehn twisting $a_3$ and $a_4$ about $V^*_{0/1}$}
\end{subfigure}
	\caption{Images of Dehn twists to aid in computations for Lemma~\ref{lem:4n+3}.}
\label{fig:twist2}
\end{figure}

\begin{proposition}\label{prop:4n+3}
The presentation $P(3,2;4/(4n+3))$ is AC-trivial.
\end{proposition}

\begin{proof}
Let $(r,s) = (r_0,s_0)$ be the relators from Lemma~\ref{lem:4n+3}, and note that we can regroup terms to express and $r_0$ and $s_0$ as
\begin{eqnarray*}
r_0 &=& \ov{y}(\ov{x}\ov{y}\ov{x})^nyx(xyx)^n; \\
s_0 &=& (\ov{x}\ov{y}\ov{x})^n\ov{x}\ov{y}\ov{x}\ov{y}(xyx)^nxy.
\end{eqnarray*}
Next, we let $r_1$ and $s_1$ be the result of cyclic permutations of $r_0$ and $s_0$, respectively, so that
\begin{eqnarray*}
r_1 &=& (\ov{x}\ov{y}\ov{x})^nyx(xyx)^n\ov{y}; \\
s_1 &=& xy(\ov{x}\ov{y}\ov{x})^n\ov{x}\ov{y}\ov{x}\ov{y}(xyx)^n.
\end{eqnarray*}
Letting $s_2 = s_1r_1$, we have
\[ s_2 = (xy(\ov{x}\ov{y}\ov{x})^n\ov{x}\ov{y}\ov{x}\ov{y}(xyx)^n)((\ov{x}\ov{y}\ov{x})^nyx(xyx)^n\ov{y}) =  xy(\ov{x}\ov{y}\ov{x})^n\ov{x}\ov{y}(xyx)^n \ov{y}.\]
Cyclically permuting $s_2$ to get $s_3$ yields
\[ s_3 = \ov{y}xy(\ov{x}\ov{y}\ov{x})^n\ov{x}\ov{y}(xyx)^n,\]
and letting $s_4 = s_3r_1$,
\[ s_4 = (\ov{y}xy(\ov{x}\ov{y}\ov{x})^n\ov{x}\ov{y}(xyx)^n)( (\ov{x}\ov{y}\ov{x})^nyx(xyx)^n\ov{y}) = \ov{y}xy(\ov{x}\ov{y}\ov{x})^n (xyx)^n\ov{y} = \ov{y}x.\]
Finally, $s_4$ is the relator $y=x$, and so we can use AC moves to transform $r_1$ into
\[r_2 = \ov{x}^{3n}x^2x^{3n}\ov{x} = x.\]
It follows that $P(3,2;4/(4n+3))$ is AC-trivial.
\end{proof}

\section{Conclusion}\label{sec:quest}

We conclude with a couple of questions to motivate future research in this area.  Recall from the introduction that the GPRC and the Andrews-Curtis conjecture are closely connected.  We have simplified the group presentations, but the motivation was the following related question about the related links.

\begin{question}
Are the links $L(3,2;4/d)$ handle-slide equivalent or stably equivalent to unlinks?
\end{question}

One way to answer this question would be to show that the trisections $\T(4/d)$ are standard and invoke Theorem 3 from~\cite{MZDehn} (for a definition of a \emph{standard} trisection, see~\cite{MZDehn}.)

\begin{question}
Are the trisections $\T(4/d)$ standard?
\end{question}

Notably, it remains open whether there exists a nonstandard trisection of $S^4$.

Finally, it would be interesting to understand whether these techniques can be applied to other families of group presentations, in particular, because in forthcoming work, Meier and the second author are able to show an number of unexpected equivalences between $P(p_1,q_1;c_1/d_1)$ and $P(p_2,q_2;c_2/d_2)$ for various parameters~\cite{MZFut}.

\begin{question}
Can these techniques be extended to show that any presentations of the form $P(5,2;4/d)$ are AC trivial?  What about presentations of the form $P(p,2;4/d)$?
\end{question}

\bibliographystyle{amsalpha}
\bibliography{actrivial}

\end{document}